\documentclass{amsart}

\usepackage[english]{babel}

\usepackage[letterpaper,top=2cm,bottom=2cm,left=3cm,right=3cm,marginparwidth=1.75cm]{geometry}

\usepackage{amsmath}
\usepackage{amssymb}
\usepackage{graphicx}
\usepackage[pdftex]{pict2e}
\usepackage{enumitem}
\usepackage{tikz}

\newtheorem{theorem}{Theorem}[section] 
\newtheorem{lemma}[theorem]{Lemma}     

\newtheorem{proposition}[theorem]{Proposition}

\theoremstyle{definition}
\newtheorem{example}[theorem]{Example}
\newtheorem{remark}[theorem]{Remark}

\newcommand\lu{\mathbb{L}}

\newcommand\alge{\mathbb{M}}
\newcommand\tr{\textnormal{ tr}}

\newcommand\oink{\mathcal O}

\newcommand\Ba{\mathbf B}
\newcommand\Ra{\mathbf R}
\newcommand\Da{\mathbf{D}}
\newcommand\Ha{\mathbf{H}}
\newcommand\Ea{\mathbf{E}}

\newcommand\cs{{}_{\clubsuit}}

\newcommand\bmattrix[4]{\left(\begin{array}{cc}#1&#2\\#3&#4\end{array}\right)}
\newcommand\sbmattrix[4]{\textnormal{\scriptsize$\left(\begin{array}{cc}#1&#2\\#3&#4\end{array}\right)$\normalsize}}

\title[Two characterizations of Bass orders]
{Two characterizations of Bass orders via branches}

\author{Luis Arenas-Carmona}

\begin{document}

\footnote{\emph{Keywords:}
Bass orders, maximal orders, Bruhat-Tits trees.\\
\emph{MSC (2020):} 11S45(Primary), 16H10, 16G30 (Secondary).}

\begin{abstract}
It has been known for some time that the orders in the four
dimensional matrix algebra over a local field
that can be written as a finite
intersection of maximal orders are precisely those whose
Gorenstein closure is Eichler. In this paper, a similar
characterization is given for orders whose Gorenstein closure
is a Bass order. A second characterization, this time for
the Bass orders themselves, is given in terms of their
branches, i.e., maximal subgraphs of the Bruhat-Tits tree
whose vertices are orders containing them.
\end{abstract}

\maketitle

\section{Introduction}\label{intro}

In all of this work, $\alge$ denotes the matrix algebra
$\mathbb{M}_2(K)$ over a local field $K$ with ring of integers
$\mathcal{O}_K$. Gorenstein orders can be defined as maximal rank (or full) orders 
$\Ra$ whose codifferent, given by
$$\mathrm{codiff}(\Ra)=\{a \in\alge| \tr(a\Ra)\subseteq\oink_K\}
\subseteq\alge,$$
is invertible as a two-sided fractional $\Ra$-ideal.
There is a list of properties that can be shown to be equivalent
to being Gorenstein. For instance, $\Ra$ is Gorenstein if
and only if $\mathrm{Hom}_{\oink_K}(\Ra,\oink_K)$ is projective
as an $\Ra$-bimodule. See \cite[Prop. 24.2.3]{voight} for
details. It is shown in \cite[Prop. 24.2.15]{voight} that for 
every full order $\Ra$ in a quaternion algebra, there exists a
unique Gorenstein order $\tilde{\Ra}=\mathrm{Gor}(\Ra)$ 
and a unique ideal $I\subseteq\oink_K$ satisfying 
\begin{equation}\label{gor}
    \Ra=\oink_K\mathbf{1}+I\tilde{\Ra},
\end{equation} 
where $\mathbf{1}=\mathbf{1}_{\alge}\in\alge$ 
is the identity matrix.
In particular, $\Ra$ is Gorenstein if and only if 
$I=\oink_K$ and $\tilde{\Ra}=\Ra$. Equivalently, a 
Gorenstein order is one for which no non-trivial expression 
of the form (\ref{gor}), for an order $\tilde{\Ra}$ and an ideal
$I$, is possible. This equivalent definition 
is frecuently used in what follows.
In current literature the order $\tilde{\Ra}$
is called the Gorenstein closure of $\Ra$,
although the reader must be warned that $\tilde{\Ra}$
is not, in general, the smallest Gorenstein order
containing $\Ra$. The term ``Gorenstein saturation''
is also used in literature.
An order $\Ba$ is Bass when every order $\Ra$ containing
$\Ba$ is Gorenstein. Both Gorenstein and Bass orders play
a central role in the classification of quaternion orders.
See \cite{Brz82}, \cite{Brz83} and \cite[\S24.2]{voight}
for more details on the general theory.

The set $\mathfrak{t}(K)$ of maximal orders in $\alge$
can be endowed with a graph structure by defining 
a neighbor relation on this set. This graph is called the
Bruhat-Tits tree. The precise definition is recalled in next 
section. In \cite{Tu}, finite intersections of maximal orders in
$\alge$ are fully characterized in terms of this graph. 
It is proven there
that any finite intersection $\Ra=\bigcap_{i=1}^n\Da_i$ 
equals an intersection of 
at most 3 maximal orders. Then either:
\begin{enumerate}
    \item $\Ra$ is an Eichler order, i.e.,
    either maximal or the intersection of two 
    maximal orders, or
    \item $\Ra$ is the intersection of three maximal orders
    corresponding to non-collinear vertices in the Bruhat-Tits
    tree.
\end{enumerate}
 In the latter case, $\Ra$ can be proved to be a 
 non-Gorenstein order whose
 Gorenstein closure is Eichler. 
 See Example \ref{e24} below
 for details. The description of these orders as intersections
 play a significant role when computing spinor images, 
 a mayor step in solving the embedding problem, i.e.,
 determining the set of orders in a given genus
 containing suborders in a particular isomorphism class.
 Since an Eichler order $\Ea$ is completely determined by the
set $\mathfrak{s}_K(\Ea)$ of maximal orders containing it, 
we can study the set of Eichler orders containing a fixed
suborder $\Ha$ by computing the set of maximal orders
containing $\Ha$ and applying the graph structure. 
 Not having an equivalent tool for other Bass orders
 makes computing spinor images for them harder.
Our purpose in this paper is finding a tool that fills
this gap. We believe that the theory of ghost intersections
described below is that tool.
 
 We define a ghost intersection of maximal orders as an 
order of the form $\Ra=\alge
\cap\bigcap_{i=1}^n\Da_i$ where
$\Da_1,\dots,\Da_n$ are $\oink_L$-maximal orders
in $\alge_L=\mathbb{M}_2(L)$ for some finite extension of fields $L/K$.
The main result of this work is a characterization of
Bass orders along the same lines as in the reference:
\begin{theorem}\label{t12}
An full order $\Ra$ is a ghost intersection 
of maximal orders if and
only if its Gorenstein closure is a Bass order.
\end{theorem}

In these terms, we can give a purely geometrical
characterization of Bass orders that is essential in the sequel
and we record it here as it is an important result on its 
own right:

\begin{theorem}\label{t11}
An order $\Ha$ is Bass if and only if the set 
$\mathfrak{s}_K(\Ha)$ of maximal orders
containing $\Ha$ is a line as a subgraph of
$\mathfrak{t}(K)$.
\end{theorem}

The importance of Theorem \ref{t12} is that we can use graphs
over field extensions of $K$ to study Bass orders, pretty much
in the same fashion as we have done for Eichler orders in
previous works. For instance, the
set of  Bass orders containing
a particular matrix $\mathbf{m}$ can be studied by computing 
the branch of the order $\oink_L[\mathbf{m}]$ in the
Bruhat-Tits tree $\mathfrak{t}(L)$, for suitable extensions
$L/K$. In fact, the proofs in this work show that we can
restrict ourselves to working on quadratic extensions.

\begin{remark}
    Note that quaternion algebras and 
    their orders can also 
    be studied through the theory of 
    ternary quadratic forms. 
    Specifically, orders can be 
    associated to quadratic forms
    which are essentially their 
    intersection with the space of 
    pure quaternions. For instance,
    in the non-dyadic case, Bass orders
    are associated to lattices of the form
    $\langle1\rangle\perp\langle b\rangle
    \perp\langle c\rangle$ with $b$ of valuation
    $0$ or $1$, see \cite[Prop. 5.8]{Lem}. 
    This relation has been used before to
    study embeddings into Eichler orders,
    see \cite{Chan}. We believe that
    the geometric method presented here have some
    advantages for ease of computation. 
    For instance, the dyadic case can be treated
    more uniformly, while the formulas
    in \cite[\S5]{Lem} often consider
    only dyadic local fields that have
    $2$ as a uniformizer.
    Examples  \ref{e64}-\ref{e69}
    in the last section illustrate 
    how our method work in explicit 
    computations.
\end{remark}

\section{Trees and branches}

In this work, by a graph $\mathfrak{g}$ we mean a
set $V=V_{\mathfrak{g}}$ whose elements are called vertices,
together with a symmetric relation 
$A_{\mathfrak{g}}=A\subseteq V\times V$,
whose elements are called edges. We assume $(v,v)$ is never 
an edge. If $e=(v,w)\in A$, then $\bar{e}=(w,v)$ is called
the reverse edge. The valency $\mathrm{val}(v)$ of 
a vertex $v$ is the number of edges $(v,w)\in A$. 
The realization $\mathrm{real}(\mathfrak{g})$ is 
the topological space obtained by identifying  the endpoints $(0,a)$ and $(1,a)$ of a marked interval
$\tilde{a}=[0,1]\times\{a\}$ to the vertices $v$ and $w$
of every pair of the form $a=\{e,\bar{e}\}$.
A subgraph $\mathfrak{h}$ of a graph $\mathfrak{g}$ is a 
subset  $V_{\mathfrak{h}}\subseteq V_{\mathfrak{g}}$ provided
with the induced relation $A_{\mathfrak{h}}=
A_{\mathfrak{g}}\cap(V_{\mathfrak{h}}\times V_{\mathfrak{h}})$.
Then we can consider $\mathrm{real}(\mathfrak{h})$ as a 
subspace of $\mathrm{real}(\mathfrak{g})$. 
An important example of graph is the real line 
graph $\mathfrak{r}$, whose realization is homeomorphic to 
$\mathbb{R}$, and whose vertex set corresponds to
$\mathbb{Z}$. We write $v_n$ for the vertex corresponding 
to $n\in\mathbb{Z}$. An integral interval is a 
connected subgraph of $\mathfrak{r}$. 
Vertices of valency one in an integral interval are 
called  endpoints.
We write $\mathfrak{i}_{n,m}$ for the integral interval with
endpoints $v_n$ and $v_m$. Infinite intervals like
$\mathfrak{i}_{0,\infty}$ or $\mathfrak{i}_{-\infty,\infty}
=\mathfrak{r}$ are defined analogously. 

A tree is a graph whose realization
is simply connected. A line in a tree is a subgraph isomorphic 
to an integral interval. Depending on the interval, we speak 
of finite lines, rays or maximal lines, the latter being
isomorphic to $\mathfrak{r}$. We also use the term endpoint
for a vertex of valency one in a line.
For every pair of vertices $\{v,w\}$ in a tree,
there is a unique finite line with $v$ and $w$ as
endpoints. Two rays in a tree are said to have the 
same visual limit
if their intersection is a ray. Visual limits are
equivalent classes under this relation. In a tree,
there is a unique ray joining a vertex to a visual limit,
meaning a unique ray in a given equivalence class
with a given initial vertex. likewise,
there is a unique maximal line joining two visual limits. See Fig. 1 for examples.
A vertex $v$ in a tree is a leaf if $\mathrm{val}(v)=1$,
a bridge if $\mathrm{val}(v)=2$, and a node if
$\mathrm{val}(v)\geq3$.

Most graphs we use are contained in the Bruhat-Tits tree
$\mathfrak{t}(K)$ for a local field $K$. The vertices
of $\mathfrak{t}(K)$ are the maximal orders $\Da$
in the matrix algebra $\alge$.
Two such orders $\Da$ and $\Da'$ are neighbors if, in some
basis, they have the form
\begin{equation}\label{nba}
\Da=\bmattrix{\oink_K}{\oink_K}{\oink_K}{\oink_K},\qquad\Da'=
\bmattrix{\oink_K}{\pi_K^{-1}\oink_K}{\pi_K\oink_K}{\oink_K}.
\end{equation}
Note that this is indeed a symmetric relation.
It is well known that this graph turns out to be
a homogeneous tree with vertices of valency
$q+1=|\oink_K/\pi_K\oink_K|+1$, and
whose visual limits are in correspondence with
the $K$-points of the projective line,
see \cite[\S II.1]{trees}.

For any order $\Ha$, not necessarily of maximal rank,
the branch $\mathfrak{s}_K(\Ha)$ is defined as the 
subgraph whose vertices are precisely the maximal orders
containing $\Ha$. The branch 
$\mathfrak{s}_K(\mathbf{u})$ for a 
matrix $\mathbf{u}\in\alge$ is defined 
analogously. Note that 
$\mathfrak{s}_K(\mathbf{u})$ is non-empty
precisely when the minimal polynomial of
$\mathbf{u}$ has integral coefficients.
In this case we say that $\mathbf{u}$
is integral. Note  that
\begin{equation}\label{gens}
\mathfrak{s}_K(\Ha)=\bigcap_{i=1}^n\mathfrak{s}_K(\mathbf{u}_i),
\end{equation}
for any set $\{\mathbf{u}_1,\dots,\mathbf{u}_n\}$ 
that generates $\Ha$
as a $\oink_K$-order. The branches $\mathfrak{s}_K(\Ha)$ 
and $\mathfrak{s}_K(\mathbf{u})$ are known to be connected,
see for example \cite[Prop. 2.3]{Eichler2}.

To compute the branch $\mathfrak{s}_K(\Ha)$ we use
relation (\ref{gens}) above. Furthermore,
it can be assumed that none of the generators $\mathbf{u}_i$
is central in $\alge$, since central integers are contained in 
every maximal order. To compute the 
branch $\mathfrak{s}_K(\mathbf{u})$, for a non central 
integral element $\mathbf{u}$, we note that the 
order $\oink_K[\mathbf{u}]$, generated by $\mathbf{u}$, 
has rank two. 
It spans the algebra $\lu=K[\mathbf{u}]$, which is 
isomorphic to one of the following: 
$K[x]/(x^2)$, $K\times K$, or a quadratic extension
$L$ of $K$. Case by case, computations are as 
follows:
\begin{enumerate}
[label=($\diamondsuit$\bf{\arabic*})]
    \item\label{dia1} If $\lu
    \cong K[x]/(x^2)$, then
    $\mathfrak{s}_K(\mathbf{u})$ is what 
    we call an infinite
    foliage. It can be defined as the union of a 
    sequence of balls $\mathfrak{b}_i$, of radius 
    $i$ and center $v_i$, for
    $i=1,2,\dots$, where $v_0,v_1,v_2,\dots$
    is the ordered sequence of vertices in a ray,
    as depicted in Fig. \ref{f1}.\textbf{I}.
     See \cite[Prop. 4.4]{Eichler2}.
    \item\label{dia2} If 
    $\lu\cong K\times K$, then
    its integral elements are all in a ring
    $\oink_\lu\cong \oink_K\times \oink_K$,
    and $\oink_K[\mathbf{u}]=\oink_K\mathbf{1}+
    \pi_K^r\oink_\lu$,
    for some integer $r\geq0$ (c.f. 
    \cite[Lem. 4.1]{Eichler2}). Then 
    $\mathfrak{s}_K(\mathbf{u})$ is
    the set of orders at distance $r$ or less
    from a maximal line, called the stem, 
    as depicted in Fig. \ref{f1}.\textbf{II}.
     See \cite[Prop. 2.4]{Eichler2}
    and \cite[Prop. 4.2]{Eichler2}.
    \item \label{dia3} If $\lu\cong L$ 
    is a quadratic extension, then again 
    we have
    $\oink_K[\mathbf{u}]=\oink_K\mathbf{1}+
    \pi_K^r\oink_\lu$
    for some integer $r\geq0$ 
    (c.f. \cite[Lem. 4.1]{Eichler2}). Then 
    $\mathfrak{s}_K(\mathbf{u})$ is
    the set of orders at distance $r$ or less
    from an edge (Fig. \ref{f1}.\textbf{III}), 
    if $L/K$ is ramified,
    or a unique vertex otherwise 
    (Fig. \ref{f1}.\textbf{IV}).
     See \cite[Prop. 2.4]{Eichler2}
    and \cite[Prop. 4.2]{Eichler2}.
\end{enumerate}
\begin{figure}
    \centering
    \setlength{\unitlength}{2mm}
    \begin{picture}(100,15)(0,-3)
        \put(2.5,-0.1){$\dots$}
        \put(1.7,-0.3){$\star$}
        \put(15,8){\textbf{I}}
        \put(5,0){\line(1,0){15}}
        \put(8,0){\line(0,1){10}}
        \put(8,5){\line(1,0){5}}
        \put(8,8){\line(1,0){2}}
        \put(11,5){\line(0,1){2}}
        \put(15,0){\line(0,1){5}}
        \put(15,3){\line(1,0){2}}
        \put(18,0){\line(0,1){2}}
        \put(7.5,-0.5){$\bullet$}
        \put(7.5,-2){$v_3$}
        \put(14.5,-0.5){$\bullet$}
        \put(14.5,-2){$v_2$}
        \put(17.5,-0.5){$\bullet$}
        \put(17.5,-2){$v_1$}
        \put(19.5,-0.5){$\bullet$}
        \put(19.5,-2){$v_0$}
        \put(34,8){\textbf{II}}
        \put(24.5,-0.1){$\dots$}
        \put(23.7,-0.3){$\star$}
        \put(41.5,-0.1){$\dots$}
        \put(43.5,-0.3){$\star$}
        \put(27,0){\line(1,0){14}}
        \put(30,0){\line(0,1){5}}
        \put(30,3){\line(1,0){2}}
        \put(34,0){\line(0,1){5}}
        \put(34,3){\line(1,0){2}}
        \put(38,0){\line(0,1){5}}
        \put(38,3){\line(1,0){2}}
        \put(29.5,-0.5){$\bullet$}
        \put(33.5,-0.5){$\bullet$}
        \put(37.5,-0.5){$\bullet$}
        \put(54,8){\textbf{III}}
        \put(47,0){\line(1,0){14}}
        \put(49,0){\line(0,1){2}}
        \put(52,0){\line(0,1){5}}
        \put(52,3){\line(1,0){2}}
        \put(56,0){\line(0,1){5}}
        \put(56,3){\line(1,0){2}}
        \put(59,0){\line(0,1){2}}
        \put(51.5,-0.5){$\bullet$}
        \put(55.5,-0.5){$\bullet$}
        \put(69,8){\textbf{IV}}
        \put(65,0){\line(1,0){10}}
        \put(67,0){\line(0,1){2}}
        \put(70,0){\line(0,1){5}}
        \put(70,3){\line(1,0){2}}
        \put(73,0){\line(0,1){2}}
        \put(69.5,-0.5){$\bullet$}
    \end{picture}
    \caption{An infinite foliage (\textbf{I}).
    The $2$-tubular neighborhood of a maximal path
    (\textbf{II}), and edge (\textbf{II}) or a
    vertex (\textbf{IV}). The stem vertices 
    for the last three are denoted by bullets.
    We assume the residue field has 2 elements.
    Visual limits are denoted by stars.}
    \label{f1}
\end{figure}
 In \ref{dia3}, the central edge
or vertex is also called stem for the sake of uniformity.
Stem vertices can be characterized as those lying
the farthest from the set of leaves. The infinite foliage 
has no stem. The ray with v\'ertices
$v_0,v_1,v_2,\dots$ in \ref{dia1} is not special,
as there is a ray like this starting from every
leaf in the infinite foliage.
Another important property, which is critical for us
and can be found in \cite[Prop. 2.4]{Eichler2},
is given in the following proposition:

\begin{proposition}\label{p21}
    Let $\Ha$ be an arbitrary order in $\alge$, and consider
    the order $\Ha^{[r]}=\oink_K\mathbf{1}+\pi_K^r\Ha$. 
    Then the branch
    $\mathfrak{s}_K\left(\Ha^{[r]}\right)$ contains precisely
    the maximal orders at distance $r$ or less from the
    branch $\mathfrak{s}_K(\Ha)$ 
    of $\Ha$.\qed
\end{proposition}

Any order that can be written as an intersection of maximal orders can be written as the intersection of
the maximal orders containing it. In other words, if
$\Ha$ is an intersection of maximal orders, then
$\Ha=\Ha_{\mathfrak{s}}:=
\bigcap_{\Da\in\mathfrak{s}}\Da$, where
$\mathfrak{s}=\mathfrak{s}_K(\Ha)$. The map
$\mathfrak{s}\mapsto\Ha_{\mathfrak{s}}$
defines a bijection between subgraphs that
are branches of orders on one hand and intersections
of maximal orders on the other. The intersection
$\Ha_{\mathfrak{s}}$ is the largest order whose branch
is $\mathfrak{s}$. For a full order $\Ra$, and for any
maximal order $\Da$ containing $\Ra$, there is an integer
$r$ satisfying $\Da^{[r]}\subseteq \Ra$, and therefore
$\mathfrak{s}_K(\Ra)\subseteq
\mathfrak{s}_K\left(\Da^{[r]}\right)$.
Hence, there is a finite
number of maximal orders $\Da'$ satisfying $\Ra\subseteq\Da'$,
as the graph $\mathfrak{s}_K\left(\Da^{[r]}\right)$ is a ball
of radius $r$ by Prop. \ref{p21}.
In particular, a full order is an intersection of maximal
orders if and only if it is a finite intersection of
maximal orders.

\begin{example}
    An Eichler order $\Ea=\Da_1\cap\Da_2$ is contained precisely
    in the maximal orders in the line $\mathfrak{p}$
    from $\Da_1$ to $\Da_2$. It follows that
    $\mathfrak{s}_K(\Ea)=\mathfrak{p}$ and 
    $\Ha_\mathfrak{p}=\Ea$.
\end{example}

\begin{example}\label{e24}
    If $\mathfrak{b}$ is a ball of radius $r$ in 
    $\mathfrak{t}(K)$, i.e., the set of maximal orders
    at distance $r$ or less from a single vertex $\Da$,
    then the corresponding intersection $\Ha_{\mathfrak{b}}$
    of maximal orders is the order $\Da^{[r]}$, see 
    \cite[Lem. 2.5]{Eichler2}. Choose
    three maximal orders $\Da_1,\Da_2,\Da_3\in
    \mathfrak{b}$, such that the lines from $\Da$ 
    to any pair of the 
    $\Da_i$ intersect in a single point. These vertices
    are said to be spread out in the ball. 
    Equivalently, we say that
    $\Da$ is the branching vertex
    of the convex hull of $\{\Da_1,
    \Da_2,\Da_3\}$, as in 
    \cite[Def. 11]{Tu}.
    Then the intersection $\Da_1\cap\Da_2\cap\Da_3$
    equals the intersection $\Ha_{\mathfrak{b}}$
    by \cite[Thm. 8]{Tu}. It follows that 
    $\Da^{[r]}=\Da_1\cap\Da_2\cap\Da_3$.
    Now choose a maximal order $\Da'_2$
    beyond $\Da_2$, i.e., chosen in a way
    that the path from $\Da'_2$ to either
    $\Da_1$ or $\Da_3$ passes through $\Da_2$. Analogously, choose
    $\Da'_3$ beyond $\Da_3$.
    Then $\Da'_2$ lies at a distance
    $r_2\geq r$ from  the center $\Da$, while
    $\Da'_3$ lies at a distance
    $r_3\geq r$ from $\Da$. Let $\Da''_2$
     is the order at distance $r_2-r$ from $\Da$
     in the path from $\Da$ to $\Da'_2$, and let
     $\Da''_3$ be defined analogously. 
     Let $\mathfrak{p}$ the line from $\Da''_2$ to $\Da''_3$.
     Let $\mathfrak{s}$ be the graph containing
     precisely the vertices at distance $r$ or less from 
     $\mathfrak{p}$. See Fig. \ref{f2}.
     Then, it follows from \cite[Thm. 8]{Tu}
     that $\Da_1\cap\Da'_2\cap\Da'_3$ equals the intersection
     $\Ha_{\mathfrak{s}}$. The graph $\mathfrak{s}$
     can be seen as the union of all balls 
     $\mathfrak{b}(\hat{\Da})$ of radius $r$ and center
     $\hat{\Da}\in\mathfrak{p}$.  Recall that $\Ha_1^{[r]}\cap
     \Ha_2^{[r]}=(\Ha_1\cap\Ha_2)^{[r]}$ for any pair of orders
     $\Ha_1$ and $\Ha_2$ \cite[Prop. 2.1]{Eichler2}. This allows us to compute
     as follows:
     $$\Ha_{\mathfrak{s}}=\bigcap_{\hat{\Da}\in\mathfrak{p}}
     \Ha_{\mathfrak{b}(\hat{\Da})}=
     \bigcap_{\hat{\Da}\in\mathfrak{p}}
     \hat{\Da}^{[r]}=
     \left(\bigcap_{\hat{\Da}\in\mathfrak{p}}
     \hat{\Da}\right)\!\!\!\!\!\!\!\!
     {\phantom{\int}}^{[r]}=
     \Ha_{\mathfrak{p}}^{[r]}=\Ea^{[r]},$$
     where $\Ea=\Da''_2\cap\Da''_3$.
     Fig. \ref{f2} shows an example 
     where $r=2$, $r_2=3$ and $r_3=5$. 
\end{example}

\begin{figure}
    \centering
    \setlength{\unitlength}{2mm}
    \begin{picture}(40,15)(-5,-3)
        \put(3,0){\line(1,0){26}}
        \put(5,0){\line(0,1){2}}
        \put(8,0){\line(0,1){5}}
        \put(8,3){\line(1,0){2}}
        \put(12,0){\line(0,1){5}}
        \put(12,3){\line(1,0){2}}
        \put(16,0){\line(0,1){5}}
        \put(16,3){\line(1,0){2}}
        \put(20,0){\line(0,1){5}}
        \put(20,3){\line(1,0){2}}
        \put(24,0){\line(0,1){5}}
        \put(24,3){\line(1,0){2}}
        \put(27,0){\line(0,1){2}}
        \put(11.5,-0.5){$\bullet$}
        \put(15.5,-0.5){$\bullet$}
        \put(19.5,-0.5){$\bullet$}
        \put(15.5,2.5){$\bullet$}
        \put(19.5,2.5){$\bullet$}
        \put(23.5,2.5){$\bullet$}
        \put(19.5,4.5){$\bullet$}
        \put(21.5,2.5){$\bullet$}
        \put(19.5,-2){${}^{\Da}$}
        \put(23.5,-0.5){$\bullet$}
        \put(26.5,-0.5){$\bullet$}
        \put(28.5,0){$\cs$}
        \put(26.5,2){$\cs$}
        \put(25.5,3){$\cs$}
        \put(23.5,5){$\cs$}
        \put(17.5,3){$\cs$}
        \put(15.5,5){$\cs$}
        \put(13.5,3){$\cs$}
        \put(11.5,5){$\cs$}
        \put(9.5,3){$\cs$}
        \put(7.5,5){$\cs$}
        \put(7.5,3){$\cs$}
        \put(11.5,3){$\cs$}
        \put(7.5,3){$\cs$}
        \put(11.5,3){$\cs$}
        \put(7.5,0){$\cs$}
        \put(4.5,0){$\cs$}
        \put(2.5,0){$\cs$}
        \put(4.5,2){$\cs$}
        \put(19.5,5){${}^{\Da_1}$}
        \put(26.5,-2){${}^{\Da_2}$}
        \put(23.5,-2){${}^{\Da''_2}$}
        \put(29.5,-1){${}^{\Da'_2}$}
        \put(11.5,-2){${}^{\Da_3}$}
        \put(7.5,-2){${}^{\Da''_3}$}
        \put(0.5,-1){${}^{\Da'_3}$}
    \end{picture}
    \caption{The bullets correspond to the maximal orders
    containing $\Da_1\cap\Da_2\cap\Da_3$.
    the clubs denote the additional orders
    containing $\Da_1\cap\Da'_2\cap\Da'_3$.
    We assume $q=2$.}
    \label{f2}
\end{figure}

Theorem 8 in \cite{Tu} shows that every intersection of maximal order that is not Eichler can be obtained in this way. More precisely, the orders $\Da_1$, $\Da'_2$ and  $\Da'_3$ only need to maximize the sum $r+r_2+r_3$, which is denoted $\frac12d_3(S)$ in \cite[\S4]{Tu}. The following two results
are straightforward consequences of this fact:

\begin{proposition}\label{p22n}
    A full order is a finite intersection of maximal orders precisely when its Gorenstein closure is Eichler.\qed
\end{proposition}

\begin{proposition}\label{p22}
    The branch $\mathfrak{s}_K(\Ra)$ 
    of every full order 
    $\Ra$ consists of
    the maximal orders at distance 
    $r\geq0$ or less from a 
    finite line $\mathfrak{p}$ called 
    the stem of the
    branch.\qed
\end{proposition}

The subgraph of $\mathfrak{t}(K)$ 
consisting precisely of
the orders at distance $r$ or less from a  
$\mathfrak{p}$ is denoted
$\mathfrak{p}^{[r]}$ and called the $r$-th 
tubular neighborhood of $\mathfrak{p}$, or the 
tubular neighborhood of width $r$ of 
$\mathfrak{p}$.
There is a stronger version of 
Proposition \ref{p22} that is shown 
in \cite[Prop. 5.3]{Eichler2} and 
\cite[Prop. 5.4]{Eichler2}:

\begin{proposition}\label{p23}
    The branch $\mathfrak{s}_K(\Ha)$ of 
    every order $\Ha$ 
    is either an infinite foliage, or a tubular 
    neighborhood $\mathfrak{p}^{[r]}$ of a line
    $\mathfrak{p}$, which could  be a ray
    or a maximal line. The infinite foliage
    appears precisely when $\Ha$ is generated by
    a nilpotent element.
    \qed
\end{proposition}

\paragraph{Proof of Theorem \ref{t11}.}
    If $\Ra$ is not a Bass order, it has a
    supra-order $\Ra'$ that is not Gorenstein,
    i.e., $\Ra'=\Ra_0^{[r]}$, where $r\geq1$
    is an integer and $\Ra_0$ is a full order. 
    Then $\Ra$ is contained in every maximal 
    orders at distance $r$ from the branch of
    $\Ra_0$, and therefore its branch
    is not a line.
    On the other hand, if the branch
    $\mathfrak{s}$ of $\Ra$ is not a line,
    it is a tubular neighborhood of width $r\geq1$
    of a line and we are in the case described
    in Ex. \ref{e24}. This means that $\Ha_{\mathfrak{s}}=\Ea^{[r]}$ is an supra-order that is not Gorenstein.
\qed

\section{Trees and field extensions}

If $L/K$ is an extension of local fields, and if 
$\Da$ is a maximal $\oink_K$-order in
$\alge\subseteq\alge_L$, then
$\Da_L=\oink_L\Da$ is a maximal $\oink_L$-order
in $\mathbb{M}_2(L)$.
Whenever $L/K$ is an unramified 
extension, the map $\Da\mapsto\Da_L$
identifies $\mathfrak{t}(K)$ with a subgraph 
 of $\mathfrak{t}(L)$. In the 
general case, however, we can only
identify the realization 
$\mathrm{real}\big(\mathfrak{t}(K)\big)$ with
the realization of a subgraph $\mathfrak{t}(L/K)$
of $\mathfrak{t}(L)$ where the extensions
$\Da_L$ and $\Da'_L$ of two neighboring vertices
$\Da$ and $\Da'$ of $\mathfrak{t}(K)$ lie
at a distance $e=e(L/K)$.
See \cite[\S3]{GB1} for details and
Fig. \ref{f3nn} for some examples.
\begin{figure}
    \centering
    \unitlength 1mm 
\linethickness{0.4pt}
\ifx\plotpoint\undefined\newsavebox{\plotpoint}\fi 
\begin{picture}(77,16)(0,24)
\put(11,29){\line(1,0){12}}
\put(17,29){\line(0,1){6}}
\put(34,29){\line(1,0){16}}
\put(42,29){\line(0,1){8}}
\put(10,28){$\bullet$}\put(22,28){$\bullet$}
\put(16,28){$\bullet$}\put(16,34){$\bullet$}
\put(33,28){$\bullet$}\put(49,28){$\bullet$}
\put(41,28){$\bullet$}\put(41,36){$\bullet$}
\multiput(12,24.5)(.0410156,.0332031){16}{\line(1,0){.0410156}}
\multiput(13.25,25.6)(.0410156,.0332031){16}{\line(1,0){.0410156}}
\multiput(14.5,26.7)(.0410156,.0332031){16}{\line(1,0){.0410156}}
\multiput(15.75,27.8)(.0410156,.0332031){16}{\line(1,0){.0410156}}
\multiput(17.25,28.3)(.0410156,-.0332031){16}{\line(1,0){.0410156}}
\multiput(18.5,27.2)(.0410156,-.0332031){16}{\line(1,0){.0410156}}
\multiput(19.75,26.1)(.0410156,-.0332031){16}{\line(1,0){.0410156}}
\multiput(21,25)(.0410156,-.0332031){16}{\line(1,0){.0410156}}
\put(10.5,23){$*$}\put(21.5,23){$*$}
\put(42,33){\line(1,0){1}}
\put(44,33){\line(1,0){1}}
\put(46,33){\line(1,0){1}}
\put(47,32.25){$*$}\put(41,32.25){$*$}
\put(38,29){\line(0,-1){1}}
\put(38,27){\line(0,-1){1}}
\put(38,25){\line(0,-1){1}}
\put(37,28){$*$}\put(37,23){$*$}
\put(46,29){\line(0,-1){1}}
\put(46,27){\line(0,-1){1}}
\put(46,25){\line(0,-1){1}}
\put(45,28){$*$}\put(45,23){$*$}
\put(62,32){\line(1,0){15}}
\put(61,31){$\bullet$}\put(76,31){$\bullet$}
\put(67,27){\line(0,1){1}}
\put(67,29){\line(0,1){1}}
\put(67,31){\line(0,1){1}}
\put(66,31){$*$}\put(66,26){$*$}
\put(72,27){\line(0,1){1}}
\put(72,29){\line(0,1){1}}
\put(72,31){\line(0,1){1}}
\put(71,31){$*$}\put(71,26){$*$}
\put(09,37){$\textnormal{\textbf{I}}$}
\put(33,37){$\textnormal{\textbf{II}}$}
\put(67,37){$\textnormal{\textbf{III}}$}
\end{picture}
    \caption{A piece of the realization of
    the Bruhat-Tits tree
    $\mathfrak{t}(K)$, with full lines and bullets, inside the realization of
    $\mathfrak{t}(L)$ with dashed lines
    and asterisks. In \textbf{I}, $L/K$
    is an unramified extension. In
    \textbf{II}, $L/K$ is an ramified
    quadratic extension. In
    \textbf{III}, with only one edge of 
    $\mathfrak{t}(K)$ shown,
    $L/K$ is an ramified
    Cubic extension. The dashed lines
    lie outside $\mathfrak{t}(L/K)$.}
    \label{f3nn}
\end{figure}
Recall that $e(L/K)$
denote the ramification index of the extension $L/K$.
The description of the branch 
$\mathfrak{s}_K(\oink_{\lu})$,
when $\lu$ is isomorphic to a quadratic extension
$L$ of $K$ can be refined in terms of the tree 
$\mathfrak{t}(L/K)$,
as described in the following results:

\begin{proposition}\label{p25}
    Assume $\lu\subseteq\alge$ is
    isomorphic to an unramified quadratic
    extension $L$ of $K$.
    Let $\lu_L=L\lu 
    \subseteq\alge_L$, which is
    isomorphic to $L\times L$,
    be the extension of scalars of $\lu$ to $L$, 
    and let $\Ha=\oink_{\lu_L}$ be its ring of 
    integers. Then the line $\mathfrak{s}_L(\Ha)
    \subseteq\mathfrak{t}(L)$ contains a single
    vertex of $\mathfrak{t}(L/K)$, 
    namely the vertex
    $\Da_L$, where $\Da$ is the only maximal order
    containing $\oink_{\lu}$.
\end{proposition}

\begin{proof}
    Recall that $\oink_{\lu}$ is contained in
    a unique maximal order 
    according to \ref{dia3}.
    Let $\mathbf{u}$ be any matrix satisfying
    $\oink_K[\mathbf{u}]=
    \oink_{\lu}$. Then, since $L/K$
    is an unramified extension, 
    the minimal polynomial
    of $\mathbf{u}$ has two eigenvalues in 
    $\oink_L$ with different
    images in the residue field of 
    $L$. In particular,
    $\mathbf{u}$ generates 
    $\Ha\cong\oink_L\times\oink_L$ 
    as an $\oink_L$-order, 
    so that $\Da_L$ contains
    $\Ha$. Furthermore, the 
    branch $\mathfrak{s}_L(\Ha)$,
    which is shown to be a line by 
    setting $r=0$ in \ref{dia2},
    must be preserved by the Galois 
    group $\mathrm{Gal}(L/K)$,
    since so is the algebra $\lu$. If the line
    $\mathfrak{s}_L(\Ha)$ contained 
    a second vertex
    of the form
    $\Da'_L$ where $\Da'\subseteq\alge$ is a 
    maximal $\oink_K$-order, then the 
    Galois group, 
    preserving two points in a line, 
    must act trivially,
    and therefore it must preserve both
    visual limits of the line, but 
    then those visual limits, as 
    Galois-invariant elements 
    in $\mathbb{P}^1(L)$,
    should actually belong to $\mathbb{P}^1(K)$,
    and therefore every vertex in the line  must
    belong to $\mathfrak{t}(L/K)$, so every
    other of these vertices corresponds
    to a vertex of $\mathfrak{t}(K)$. 
    This contradicts the fact that 
    $\mathfrak{s}_K(\mathbf{u})$
    has a single vertex. The result follows.
\end{proof}

\begin{proposition}\label{p26}
    Assume $\lu\subseteq\alge $
    is isomorphic to a ramified
    separable extension $L$ of $K$.
    Let $e$ be the edge in $\mathfrak{t}(K)$
    connecting the two maximal orders of $\alge $
    containing $\oink_{\lu}$.
    Let $\lu_L=L\lu\cong L\times L$
    be the extension of $\lu$, and let 
    $\Ha=\oink_{\lu_L}$ be its ring of integers. 
    Then the vertex $w$ of $\mathfrak{t}(L/K)$ whose distance
    $d_w$ to $\mathfrak{s}_L(\Ha)$ is minimal
    is the midpoint of the edge $e$. 
    Furthermore, $d_w=0$ if $K$ is non-dyadic, while
    $d_w>0$ for a dyadic field $K$.
\end{proposition}

\begin{proof}
    Again, $\oink_{\lu}$ is contained in two 
    maximal orders sharing an edge by \ref{dia3}.
    Here the matrix $\mathbf{u}$ satisfying 
    $\oink_K[\mathbf{u}]=
    \oink_{\lu}$ can be assumed to be a uniformizing
    parameter of the discrete valuation domain
    $\oink_{\lu}$. We claim that
    conjugation by $\mathbf{u}$ leaves every vertex in the line
    $\mathfrak{s}_L(\Ha)$ invariant, while it permutes
    the two maximal orders in $\alge $
    containing $\oink_{\lu}$,
    and therefore leaves no point of the realization of
    $\mathfrak{t}(L/K)$ invariant, except for the midpoint 
    $\Da'$ of $e$, which is a vertex of 
    $\mathfrak{t}(L/K)$, and therefore a maximal order
    in $\alge_L$. Conjugation by $\mathbf{u}$ 
    leaves the point of $\mathfrak{t}(L/K)$ that is
    closest to $\mathfrak{s}_L(\Ha)$ invariant, and therefore
    this point must be the vertex $\Da'$. 
    
    Next, we prove the claim. Conjugation by $\mathbf{u}$ must
    preserve the branch $\mathfrak{s}_K(\mathbf{u})$
    which coincide with the edge $e$. Furthermore, the normalizer
    of a maximal order $\Da$ is $K^*\Da^*$, which contains
    only elements whose determinant has even valuation.
    Since the determinant of the uniformizer $\mathbf{u}$
    has odd valuation, conjugation by $\mathbf{u}$ leaves
    no vertex of $\mathfrak{t}(K)$ invariant. In particular,
    it must permute the endpoints of $e$.
    On the other hand, in $\alge_L$, 
    the matrix $\mathbf{u}$
    can be factored as $\mathbf{u}=\pi_L\mathbf{u}'$, 
    where $\pi_L$ is a
    uniformizer, and $\mathbf{u}'$ is a unit in $\Ha$. It follows
    that conjugation by $\mathbf{u}$ leaves invariant every 
    maximal order containing $\Ha$.

    Now we prove the final statement. For a non-dyadic field $K$
    we can assume that $\pi_L=\sqrt{\pi_K}$ is an eigenvalue
    of $\mathbf{u}$, so the other is $-\sqrt{\pi_K}$, and 
    $\mathbf{u}'$ has eigenvalues $\pm1$. It follows that 
    the unit $\mathbf{u}'$ does not belong to the order
    $\mathcal{O}_L\mathbf{1}+
    \pi_L\Ha$, and therefore it is a generator of $\Ha$.
    Since conjugation by $\mathbf{u}'$ leaves $\Da'$
    invariant, then $\mathbf{u}'\in K^*(\Da')^*$, whence
    $\mathbf{u}'$ belongs to $\Da'$ as its determinant is a 
    unit. This implies that $\Da'$ belongs to
    $\mathfrak{s}_L(\mathbf{u}')=\mathfrak{s}_L(\Ha)$.
    Assume now that $K$ is dyadic. Then the trace 
    $\tr(\mathbf{u})\in K$ is not a unit, and therefore
    it must be divisible by $\pi_K$. It follows that
    $\tr(\mathbf{u}')$ is not a unit, whence
    its eigenvalues coincide over the residue field.
    This means that $\mathbf{u}'$ belongs to 
    $\Ha^{[1]}=\mathcal{O}_L\mathbf{1}+\pi_L\Ha$.
    If $\Da'$ is in the line $\mathfrak{s}_L(\Ha)$,
    then any neighbor belongs to $\mathfrak{s}_L(\Ha)^{[1]}
    =\mathfrak{s}_L\left(\Ha^{[1]}\right)\subseteq
    \mathfrak{s}_L(\mathbf{u}')$, and therefore
    it is invariant under conjugation by $\mathbf{u}$.
    Since $\Da'$ has neighbors in 
    $\mathfrak{t}(L/K)$ corresponding
    to vertices of $\mathfrak{t}(K)$,
    this cannot be the case. We conclude that 
    $\Da'$ is not in the line $\mathfrak{s}_L(\Ha)$.
\end{proof}

\begin{proposition}\label{p27}
    Assume $\lu\subseteq\alge $
    is isomorphic to an
    inseparable extension $L$ of $K$.
    Let $\lu_L=L\lu\cong L[x]/(x^2)$.
    Let $z$ be the visual limit of the branch of
    any non-trivial order in $\lu_L$. 
    Then the vertex of $\mathfrak{t}(L/K)$
    that is closest to $z$ is the midpoint of the edge
     connecting the two maximal orders containing 
    $\oink_{\lu}$. 
\end{proposition}

\begin{proof}
    Note that in this case $K$ has characteristic $2$, 
    and therefore it is dyadic. The extension $L/K$
    is ramified, and therefore $\oink_{\lu}$ is
    contained precisely in the endpoints of an edge by
    \ref{dia3}. Again we set $\oink_{\lu}=
    \mathcal{O}_K[\mathbf{u}]$, where $\mathbf{u}$
    corresponds to a uniformizer of $L$.
    In this case $\lu_L=L\lu\cong 
    L[x]/(x^2)$ consists only on matrices 
    of the form $\mathbf{p}=
    a\mathbf{1}+b\mathbf{e}$, where $a,b\in K$ and
    $\mathbf{e}$ is a fixed nilpotent element.
    Then $\mathbf{p}$ is integral precisely when
    $a$ is.
    In particular, all orders spanning 
    this algebra 
    have the form 
    $\mathcal{O}_L\mathbf{1}+
    \pi_L^r\mathcal{O}_L\mathbf{e}$,
    for some integer $r\in\mathbb{Z}$. The
    branches of these orders have the 
    shape described
    in \ref{dia1}, and any two of them, say $\Ha$ and $\Ha'$,
    are connected by a relation of the form $\Ha^{[t]}=\Ha'$,
    up to permutation, so in particular they have all the 
    same visual limit.
    In this case $\oink_{\lu_L}$ is not 
    defined, since $\lu_L$ has no maximal order. 
    However, we can still 
    write $\mathbf{u}=\pi_L\mathbf{u}'$, as above, 
    where $\mathbf{u}'\in\lu_L$
    generates an order $\Ha=\mathcal{O}_L[\mathbf{u}']$, 
    and it is a unit. In particular, conjugation by
    $\mathbf{u}'$ (or, equivalently, by $\mathbf{u}$) leaves
    invariant every order containing $\mathbf{u}'$.
    We prove as before that conjugation by $\mathbf{u}$ does
    not leave any vertex of $\mathfrak{t}(K)$ 
    invariant.
    In particular, the vertex $\Da'$ of
    $\mathfrak{t}(L/K)$ corresponding to the 
    midpoint
    of the edge $\mathfrak{s}_K(\mathcal{O}_{\lu})
    \subseteq\mathfrak{t}(K)$ must be a leaf of
    $\mathfrak{s}_L(\mathbf{u}')$. The result 
    follows.

\end{proof}

Although invariant visual limits are 
those with representatives in the sub-tree
$\mathfrak{t}(L/K)$, the reader must be warned
that the analogous statement on vertices is false 
in general. For instance, the maximal
orders in the smallest path joining 
$\mathfrak{s}_L(\Ha)$ and
$\mathfrak{t}(L/K)$, in the dyadic case 
of Prop. \ref{p26}, are Galois invariant.

\section{On the classification of Bass Orders}

\begin{proposition}\label{p41n}
Let $\Ba$ be a Bass order. Then precisely one of the following
conditions holds:
\begin{enumerate}[label=($\mathfrak{E}$\bf{\arabic*})]
    \item\label{Ea1} $\Ba$ is contained in a unique maximal 
    order $\Da$, and it has the form 
    $\Ba=\mathcal{O}_{\lu}+\pi_K^r\Da$,
    for some integer $r\geq0$, and for some two-dimensional 
    subalgebra $\lu\subseteq\alge$ that is
    isomorphic to an unramified quadratic extension $L/K$
    and satisfies $\mathcal{O}_{\lu}\subseteq\Da$.
    \item\label{Ea2} $\Ba$ is contained in two maximal orders 
    $\Da_1$ and $\Da_2$ sharing an edge, and it has the form 
    $\Ba=\mathcal{O}_{\lu}+
    \boldsymbol\pi_{\lu}^r(\Da_1\cap\Da_2)$, 
    for some algebra $\lu\subseteq\alge$ isomorphic 
    to a ramified quadratic extension $L/K$ satisfying
    $\mathcal{O}_{\lu}\subseteq(\Da_1\cap\Da_2)$
    and for some 
    uniformizer $\boldsymbol\pi_{\lu}$.
    \item \label{Ea3} $\Ba$ is an Eichler order 
    of level $d\geq 2$.
\end{enumerate} 
Furthermore, any full ideal $\Ba\subseteq\alge$ satisfying one
of these conditions is a Bass order.
\end{proposition}

\begin{proof}
    The last statement is immediate for Eichler orders.
    In the other two cases the branch is either a vertex or 
    an edge, so the result follows from Thm. \ref{t11}.
    To prove that Bass orders satisfy one of the conditions
    \ref{Ea1}-\ref{Ea3}, we use the classification in
    \cite[Prop. 5.4]{Brz83}. Let $\mathbf{x}_{ij}$ be
    the matrix with a coordinate $1$ in the intersection of 
    the $i$-th row and the $j$-th column, and $0$ elsewhere. 
    then every non-maximal 
    Bass order is conjugate to one of the following:
    \begin{enumerate}[label=($\mathfrak{J}$\bf{\arabic*})]
    \item\label{Ja1} The order $\Ea_n$ generated by
    $\mathbf{1}$, $\mathbf{x}_{11}$, $\mathbf{x}_{12}$
    and $\pi_K^n\mathbf{x}_{21}$, for some integer
    $r\geq1$.
    \item\label{Ja2} The order $\Ba_{n,\epsilon}$ generated by
    $\mathbf{1}$, $\mathbf{x}_{11}+
    \mathbf{x}_{12}-\epsilon\mathbf{x}_{21}$, 
    $\pi_K^n\mathbf{x}_{11}$
    and $\pi_K^n\mathbf{x}_{12}$, for some integer
    $n\geq1$, and some unit $\epsilon$ for which
    the polinomial $t^2-t+\epsilon$ 
    is irreducible.
    \item \label{Ja3} The order 
    $\Ba'_{n,p,\alpha_1\alpha_2}$ generated by
    $\mathbf{1}$, $\alpha_1\mathbf{x}_{11}+
    \alpha_2\mathbf{x}_{12}-
    \pi_K\mathbf{x}_{21}$, 
    $\pi_K^n\mathbf{x}_{11}$
    and $\pi_K^{n+p}\mathbf{x}_{21}$, 
    for some unit 
    $\alpha_2$, some element $\alpha_1$
    with positive valuation and some 
    integers $n\geq1$,
    $p\in\{0,1\}$.
\end{enumerate} 
    We need to prove the statement case 
    by case:
    \begin{itemize}
    \item The order
    $\Ea_n$ in case \ref{Ja1} is Eichler, since
    it is the intersection of the two 
    maximal orders $\Da_0=
    \bmattrix{\oink_K}{\oink_K}{\oink_K}{\oink_K}$
    and $\Da_n=
    \bmattrix{\oink_K}{\pi_K^{-n}\oink_K}
    {\pi_K^n\oink_K}{\oink_K}$. This 
    means we are in
    case \ref{Ea3}, unless $n\leq1$, 
    and in that case
    we can simply take $r=0$ in case 
    \ref{Ea1} or \ref{Ea2}.
    \item
    Case \ref{Ja2} falls into case \ref{Ea1}.
    In fact, the matrix $\mathbf{u}=
    \mathbf{x}_{11}+\mathbf{x}_{12}-
    \epsilon\mathbf{x}_{21}$ generates 
    the ring of integers
    in an algebra $\lu\subseteq\alge$ that is 
    isomorphic to an unramified extension, and 
    therefore
    it is contained in a unique maximal order, 
    which must be
    $\Da_0$, since $\mathbf{u}$ has integral 
    coefficients. Since $\pi_K^n\mathbf{x}_{22}=
    \pi_K^n\mathbf{1}-\pi_K^n\mathbf{x}_{11}$ and
    $\pi_K^n\mathbf{x}_{12}=
    \pi_K^n\mathbf{u}-\pi_K^n\mathbf{x}_{11}
    +\pi_K^n\epsilon\mathbf{x}_{21}$, we conclude
    that $\pi_K^n\Da_0\subseteq\Ba_{n,\epsilon}$,
    whence the result follows.
    \item
    In case \ref{Ja3}, the matrix 
    $\mathbf{p}=
    \alpha_1\mathbf{x}_{11}+
    \alpha_2\mathbf{x}_{12}-
    \pi_K\mathbf{x}_{21}$ 
    has the minimal polynomial
    $t^2-\alpha_1t+\pi_K\alpha_2$, 
    which is an Eisenstein 
    polynomial, and therefore it generates 
    the ring of integers
    in an algebra $\lu\subseteq\alge$ that is 
    isomorphic to an unramified extension $L/K$. 
    In particular, it is contained precisely in 
    two maximal orders sharing an edge.
    Since it is contained in $\Da_0$ and $\Da_1$,
    these are the unique two maximal orders containing it.
    Now we note that $\frac{\mathbf{p}^2}{\pi_K}=
    \frac{\alpha_1}{\pi_K}\mathbf{p}-\alpha_2\mathbf{1}$
    belongs to $\Ea=\Da_0\cap\Da_1$, and has eigenvalues that
    are units in $L$. It follows that $\frac{\mathbf{p}^2}
    {\pi_K}$ is a unit in $\Ea$, and therefore 
    $\mathbf{p}^2\Ea=\pi_K\Ea$. Since $\Ea$ is spanned by
    $\mathbf{1}$, $\mathbf{p}$, $\mathbf{x}_{11}$
    and $\pi_K\mathbf{x}_{21}$, this proves that
    $\oink_{\lu}+\mathbf{p}^{2r}\Ea=
    \Ba'_{r,1,\alpha_1,\alpha_2}$ for $r\geq0$.
    On the other hand, since
    $\mathbf{p}^2=
    \alpha_1\mathbf{p}-\pi_K\alpha_2\mathbf{1}$,
    we have that
    $$\mathbf{p}\Ea=
    \langle\mathbf{p},\mathbf{p}^2,
    \mathbf{p}\mathbf{x}_{11},
    \pi_K\mathbf{p}\mathbf{x}_{21}
    \rangle_{\oink_K}=
    \langle\mathbf{p},\pi_K\alpha_2\mathbf{1},
    \mathbf{p}\mathbf{x}_{11},
    \pi_K\mathbf{p}\mathbf{x}_{21}
    \rangle_{\oink_K}.$$
    Since $\alpha_2$ is a unit and we have 
    $\mathbf{x}_{12}=\alpha_2^{-1}
    \left(\mathbf{p}-\mathbf{p}
    \mathbf{x}_{11}\right)$, the above identity gives
    $$\mathbf{p}\Ea=\langle\mathbf{p},
    \pi_K\mathbf{1},\mathbf{x}_{12},
    \pi_K\mathbf{p}\mathbf{x}_{21}
    \rangle_{\oink_K}.$$
    Finally, the identities
    $\mathbf{p}\mathbf{x}_{21}=
    \alpha_2\mathbf{x}_{11}$ and
    $\pi_K\mathbf{x}_{21}=\alpha_2\mathbf{x}_{12}+
    \alpha_1\mathbf{x}_{11}-\mathbf{p}$
    give
    $$\mathbf{p}\Ea=
    \langle\mathbf{p},\pi_K\mathbf{1},
    \mathbf{x}_{12},
    \pi_K\mathbf{x}_{11}\rangle_{\oink_K}
    =\langle\mathbf{p},
    \pi_K\mathbf{1},\pi_K\mathbf{x}_{21},
    \pi_K\mathbf{x}_{11}
    \rangle_{\oink_K}.$$
    We conclude that $$\oink_{\lu}+
    \mathbf{p}^{2r+1}\Ea=
    \langle\mathbf{p},
    \mathbf{1}\rangle_{\oink_K}+
    \pi_K^r\mathbf{p}\Ea=
    \Ba'_{r+1,0,\alpha_1,\alpha_2},$$
    for $r\geq0$.
    Since $\Ba'_{0,0,\alpha_1,\alpha_2}=\Da_0$,
    the result follows.
    \end{itemize}
\end{proof}

\begin{example}
Assume $K$ is a non-dyadic local field.
Consider the tree in Fig. \ref{nbh1}, where
the isomorphism classes of a Bass order
$\Ba$ correspond to a vertex $v_{\Ba}$, 
and an edge connects it to $v_{\Ba'}$
for any order $\Ba'$ isomorphic
to a maximal suborder of $\Ba$, 
see \cite[Fig. 1]{Lem}.
Vertices at the same level correspond
to Bass orders with the same discriminant. 
      \begin{figure}[h!]
\begin{center}
\begin{tikzpicture}[scale=1]
\filldraw (1.5,2.63) circle(1.5pt);
\draw (1.5,2.63) -- ++(210:1.8);
\draw (1.5,2.63) -- ++(330:1.8);
\filldraw (0,1.73) circle(1.5pt);
\draw (0,1.73) -- ++(270:2);
\filldraw (0,0.73) circle(1.5pt);
\filldraw (0,-.27) circle(1.5pt)node[anchor=north] {$\vdots$};
\filldraw (3.05,1.73) circle(1.5pt);
\draw (3.05,1.73) -- ++(270:2);
\filldraw (3.05,1.23) circle(1.5pt);
\filldraw (3.05,.73) circle(1.5pt);
\filldraw (3.05,.23) circle(1.5pt);
\filldraw (3.05,-.27) circle(1.5pt)node[anchor=north] {$\vdots$};
\filldraw (2.26,2.18) 
circle(1.5pt)node[anchor=south] {${}_u$};
\draw (2.26,2.16) -- ++(210:.9);
\filldraw (1.5,1.73) 
circle(1.5pt)node[anchor=south] {${}_w$};
\draw (1.5,1.73) -- ++(240:.55);
\filldraw (1.21,1.23) circle(1.5pt);
\draw (1.21,1.23) -- ++(270:1.5);
\filldraw (1.21,.73) circle(1.5pt);
\filldraw (1.21,.23) circle(1.5pt);
\filldraw (1.21,-.27) circle(1.5pt)node[anchor=north] {$\vdots$};
\draw (1.5,1.73) -- ++(300:.55);
\filldraw (1.79,1.23) circle(1.5pt);
\draw (1.79,1.23) -- ++(270:1.5);
\filldraw (1.79,.73) circle(1.5pt);
\filldraw (1.79,.23) circle(1.5pt);
\filldraw (1.79,-.27) circle(1.5pt)node[anchor=north] {$\vdots$};
\end{tikzpicture}
\end{center}
\caption{The tree of isomorphism classes of Bass orders in $\alge $ when 
$K$ is non-dyadic.}\label{nbh1}
\end{figure}
The vertical line on the left corresponds to
non-maximal Bass orders containing
an unramified extension (Case \ref{Ea1}).
The line on the right corresponds to
Eichler orders of level $2$ or greater
(Case \ref{Ea3}).
The two central vertical lines correspond
to Bass orders in case \ref{Ea2}
containing either of
the ramified quadratic extensions.
The vertex marked $u$ corresponds to
an Eichler order of level one.
The vertex marked $w$ corresponds to
a non-Eichler Bass order containing copies
of either ramified extension.
The latter is also in case \ref{Ea2}.
   
\end{example}

\section{Proof of the main result}

\begin{lemma}\label{l43n}
Let $\mathcal{O}_\lu$ be the ring of 
integers in a subalgebra $\lu\cong L$
of $\alge=\alge $, where $L/K$ is a 
quadratic extension of local fields.
Let $\Ba$ be a full order containing 
$\mathcal{O}_\lu$. Let
$\boldsymbol{\pi}_\lu$ be a uniformizer
of $\mathcal{O}_\lu$.
Then any full suborder $\Ba'$ of $\Ba$ containing 
$\mathcal{O}_\lu$ has the form $\Ba'=
\mathcal{O}_\lu+\boldsymbol{\pi}_\lu^r\Ba$
for some $r\geq0$.
\end{lemma}

\begin{proof}
    Any such order $\Ba'$ is a rank two left
    $\mathcal{O}_\lu$-module. Since
    $\mathcal{O}_\lu$ is a principal ideal domain, the quotient module
    $\Ba/\mathcal{O}_\lu$ has the form
    $\mathcal{O}_\lu\oplus\mathcal{T}$,
    where $\mathcal{T}$ is a torsion module.
    Torsion elements in this quotient corresponds
    to elements in $\lu$. Since every element in $\Ba$ must be integral over 
    $\mathcal{O}_K$, we conclude that
    $\Ba\cap\lu=\oink_{\lu}$, and therefore $\mathcal{T}=
    \{0\}$, i.e., $\Ba/\mathcal{O}_\lu
    \cong\mathcal{O}_\lu$ as modules.
    It follows from the Correspondence Theorem
    for Modules that the submodules of $\Ba$
    containing $\mathcal{O}_\lu$ are
    in correspondence with the ideals of
    $\mathcal{O}_\lu$. In fact,
    the ideal $I=I\mathcal{O}_\lu$
    corresponds to 
    $I(\Ba/\mathcal{O}_\lu)$, and therefore to $\mathcal{O}_\lu+I\Ba$. The result follows if we recall that
    every ideal in $\mathcal{O}_\lu$
    is generated by a power of the uniformizer.
\end{proof}

Next result can be found in 
\cite[Cor. 3.2]{GB1}:

\begin{lemma}\label{c3x}
    If a maximal order $\Da$ is a leaf of the branch $\mathfrak{s}_K(\Ha)$ then the extension
    $\Da_L$ is a leaf of the branch 
    $\mathfrak{s}_L(\Ha_L)$ for any field extension
    $L/K$.   \qed 
\end{lemma}

\begin{proposition}\label{p45}
    Every Bass order is a ghost intersection of maximal orders.
\end{proposition}

\begin{proof}
We use the classification in last section 
to give a case
by case proof.  The result is well known for 
Eichler orders, for which the field $L$ 
on which the orders
are defined can be assumed to be $L=K$. Therefore,
we can assume that the Bass order $\Ba$ is 
in either of the cases \ref{Ea1} or 
\ref{Ea2} in Prop. \ref{p41n}.

Firstly, we let $\Ba$ be a Bass order of the form
$\Ba=\mathcal{O}_{\lu}+\pi_K^r\Da$,
where $L/K$ is an unramified quadratic extension,
and $\Da$ is the only maximal order containing
$\Ba$. Then the extension 
$\lu_L=L\lu\cong L\times L$
has a ring of integers 
$\oink_{\lu_L}\cong\oink_L\times\oink_L$,
which is contained precisely in the maximal orders 
in an infinite line $\mathfrak{p}$, passing through $\Da_L$, 
and whose visual limits $z$ and $\bar{z}$
satisfy $\bar{z}=\sigma(z)$,
as elements in $\mathbb{P}^1(L)$,
where $\sigma$ is the 
generator of the Galois group $\mathrm{Gal}(L/K)$.  
Let $\Da_1\subseteq\alge_L$ be a 
maximal order in $\mathfrak{p}$ at distance $r$ from $\Da$. Set 
$\Ba'=\Da_L\cap\Da_1\cap\alge $. 
This is an order in
$\alge $ that contains $\oink_{\lu}$ and 
it is contained
in $\Da$. It follows that 
$\Ba=\mathcal{O}_{\lu}+\pi_K^{r'}\Da$ for 
some integer
$r'$ by Lem. \ref{l43n}. In fact, $r'$ is the smallest integer $s$ satisfying
$\pi_K^s\Da\subseteq\Da_1$, which is equivalent
to $\Da_L^{[s]}\subseteq\Da_1$. Since 
$\mathfrak{s}_L\left(\Da_L^{[s]}\right)$ is a ball of radius $s$
around the vertex $\Da_L$, it follows that $r=r'$, 
and the result follows.

Finally, assume that $\Ba$ is a Bass order of the form 
$\Ba=\mathcal{O}_{\lu}+
\boldsymbol\pi_{\lu}^r\Ea$,
where $\Ea=\Da\cap\Da'$ is an 
Eichler order of level $1$
containing $\Ba$, the subalgebra 
$\lu\subseteq\alge $ is isomorphic
to the ramified quadratic 
extension $L$ of $K$,  and 
$\boldsymbol\pi_{\lu}\in\oink_{\lu}$
is a uniformizer. Let 
$\Ea_L=\oink_L\Ea$ be the extension,
which is an Eichler order, as it 
contains a non-trivial
idempotent. The order $\Ea_L$ is
contained precisely 
in the maximal orders between 
$\Da_L$ and $\Da_L'$ 
in the tree $\mathfrak{t}(L)$, 
as follows from Lem.
\ref{c3x}. In fact, $\Ea_L$ is contained
in exactly $3$ maximal orders in $\alge_L$,
since $\Da_L$ and $\Da_L'$ lie at distance $2$ in 
$\mathfrak{t}(L)$. The third order is
denoted $\hat{\Da}$.
It is located between $\Da_L$ and $\Da_L'$ in 
$\mathfrak{t}(L)$, and it is not defined over
$K$, i.e., is not the extension of a maximal 
order in $\alge$.
Note that we can write
$\boldsymbol\pi_{\lu}=
\pi_L\mathtt{u}$ in $\alge_L$, where 
$\mathtt{u}$ is a unit in 
the algebra $\lu_L=L\lu$,
as in the proof of Prop. \ref{p27}.
Assume first that $L/K$ is separable. Then
the matrix $\boldsymbol\pi_{\lu}$ has two 
eigenvectors in the vector 
space $L^2$, which correspond to visual 
limits $z$ and $\bar{z}$ 
of the tree $\mathfrak{t}(L)$, and the 
generator $\sigma$
of the Galois group $\mathcal{G}=
\mathrm{Gal}(L/K)$ 
satisfies $\sigma(z)=\bar{z}$. Let $\mathfrak{p}$
be the infinite line with visual limits 
$z$ and $\bar{z}$.
Let $\Da_0$ be the point of the  $\mathfrak{p}$
that is closest to $\mathfrak{t}(L/K)$.
By Prop. \ref{p26}, the vertex in 
$\mathfrak{t}(L/K)$ that is closest to $\Da_0$
is $\hat{\Da}$. Note that $\Da_0=\hat{\Da}$ 
precisely when $K$ is non-dyadic.
Consider a maximal order $\Da_1$ 
in the ray from $\hat{\Da}$ to $z$ at distance 
$r$ from $\hat{\Da}$. In the inseparable case
we use the ray joining $\hat{\Da}$ to 
the visual limit $z$ of the branch
$\mathfrak{q}$ of any non-trivial
order in $\lu_L$, as described in Prop.
\ref{p27}, but the reasoning is similar.
Note that $\oink_{\lu}$
is contained in every maximal order in the line $\mathfrak{p}$
(or the branch $\mathfrak{q}$),
and also in $\hat{\Da}$, so it must be contained in
$\Da_1$. We claim that
$\Ba=\Da_1\cap\Ea\cap\alge $. Again,
set $\Ba'=\Da_1\cap\Ea\cap\alge $.
Then $\Ba'$ is an order containing $\oink_{\lu}$
and contained in $\Ea$, so that 
$\Ba'=
\mathcal{O}_{\lu}+
\boldsymbol\pi_{\lu}^{r'}\Ea$,
for some positive integer $r'$, which is 
the smallest value of $s$ satisfying
$\boldsymbol\pi_{\lu}^s\Ea\subseteq\Da_1$.
Since $\Ea$ spans $\Ea_L$ as a $\oink_L$-module,
we can replace $\Ea$ by $\Ea_L$ and require that
$\boldsymbol\pi_{\lu}^s\Ea_L=
\pi_L^s\mathtt{u}^s\Ea_L\subseteq\Da_1$.
Since $\mathtt{u}\in\oink_{\lu}^*
\subseteq\Ea_L^*$, we can ignore it 
and just require
that $\pi_L^s\Ea_L\subseteq\Da_1$, which is 
equivalent to $s\geq r$ by the general theory. The result 
follows by reasoning as in the previous case.
\end{proof}

\begin{proposition}\label{p46}
    Assume $\Ra$ is a ghost intersection of maximal orders, and
    let  $\Ra'=\oink_K\mathbf{1}+\pi_K^r\Ra$. Then $\Ra'$
    is a ghost intersection of maximal orders.
\end{proposition}

\begin{proof}
    Assume $\Ra=\alge \cap\bigcap_{i=1}^n\Da_i$,
    where $\Da_1,\dots,\Da_n$ are maximal 
    $\oink_L$-orders in $\alge_L$ for
    some extension $L$ of $K$. Since 
    $\Ra$ is contained
    in some maximal $\oink_K$-order  in $\alge$,
    it does not hurt to assume that at least one 
    of the $\Da_i$ is defined over $K$, say 
    $\Da_1$ to fix ideas. Let $\Ra'':=
    \alge \cap\bigcap_{i=1}^n\Da_i^{[er]}$,
    where $e=e(L/K)$ is the ramification index.
    We claim that $\Ra'=\Ra''$. 
    Since $\Da_i^{[er]}$ is an intersection of
    maximal orders, as seen in Ex. \ref{e24},
    the result follows. 
    
    Let $\mathbf{h}\in\Ra'$. Write
    $\mathbf{h}=z\mathbf{1}+\pi_K^r\mathbf{u}$
    with $\mathbf{u}\in\Ra$ and $z\in\oink_K$.
    Since $\mathbf{u}\in\Da_i$, we have
    $\pi_K^r\mathbf{u}\in\Da_i^{[er]}$, since the
    normalized valuation $\nu_L$ on the local 
    field $L$ satisfies
    $\nu_L\left(\pi_K^r\right)=er$.
    It follows that
    $\mathbf{h}\in\Ra''$.
    
    Now assume 
    $\mathbf{h}\in\Ra''$. Write
    $\Da_1=\Da_L$, where $\Da$ is an 
    $\oink_K$-order
    in $\alge $. Since 
    $\mathbf{h}\in\Da_1^{[er]}$,
    then $\mathbf{h}\in\Da'_L$ for 
    every maximal $\oink_K$-order 
    $\Da'$ at distance $r$ from $\Da$, 
    according to Prop. 
    \ref{p21}. Since 
    $\mathbf{h}\in\alge $, then 
    $\mathbf{h}\in\Da'$
    for every such maximal $\oink_K$-order $\Da'$,
    i.e., $\mathbf{h}\in\Da^{[r]}$. 
    It follows that 
    $\mathbf{h}=z\mathbf{1}+\pi_K^r\mathbf{u}$,
    where $z\in\oink_K$ and $\mathbf{u}\in\Da$. 
    Since
    $\mathbf{h}\in\Da_i^{[er]}$, then 
    $\mathbf{h}=z'\mathbf{1}+
    \pi_K^r\mathbf{u}'$ with
    $\mathbf{u}'\in \Da_i$, but then 
    $\mathbf{u}-\mathbf{u}'=\left(
    \frac{z'-z}{\pi_K^r}\right)\mathbf{1}$. 
    Note that $\mathbf{u}'$ and $\mathbf{u}$
    commute since their difference is central
    in $\alge_L$. In particular,
    $\mathbf{u}-\mathbf{u}'$ is the difference
     of two commuting integral matrices, and 
     therefore it is integral. We conclude that
     $\frac{z'-z}{\pi_K^r}$ is in $\oink_K$. 
     Hence, $\mathbf{u}\in \Da_i$. As 
    $\mathbf{u}\in\alge $, then 
    $\mathbf{u}\in\Ra$, and
    therefore $\mathbf{h}\in\Ra'$ as  claimed.
\end{proof}

\begin{proposition}\label{p47}
    Assume $\Ra$ is a ghost intersection of 
    maximal orders, and its branch 
    $\mathfrak{s}_K(\Ra)$ contains a ball 
    $\mathfrak{b}\subseteq\mathfrak{t}(K)$ of 
    radius $r$. Then there is
    an order $\tilde{\Ra}$ that is a ghost intersection of maximal
    orders and satisfies $\Ra=\tilde{\Ra}^{[r]}$.
\end{proposition}

\begin{proof}
    Assume $\Ra=\alge \cap\bigcap_{i=1}^n\Da_i$,
    where $\Da_1,\dots,\Da_n$ are maximal $\oink_L$-orders for
    some extension $L$ of $K$. As before, it does not hurt
    to assume that at least three of these orders,
    say $\Da_1$, $\Da_2$ and $\Da_3$ are spread out
    leaves of the ball $\mathfrak{b}$, in the sense
    defined in Example \ref{e24}. It follows that the 
    branch $\mathfrak{s}_L\left(\bigcap_{i=1}^n\Da_i\right)$
    has width $er$ or larger, so we can write
    $\bigcap_{i=1}^n\Da_i=\bigcap_{j=1}^m(\Da'_j)^{[er]}$
    for some collection $\Da'_1,\dots,\Da'_m$.
    Furthermore, we can assume that one of these orders
    is the center of the ball $\mathfrak{b}$. In particular it is defined over $K$, so the proof of $\Ra=\tilde{\Ra}^{[r]}$,
    where $\tilde{\Ra}=\alge \cap
    \bigcap_{j=1}^m\Da'_j$, can be
    carried over from last proposition.
    The result follows.
\end{proof}

\begin{proof}[Proof of Theorem \ref{t12}]
Let $\Ra$ be an order that is a
ghost intersection 
of maximal orders. If it is Bass,
there is nothing to prove.
If $\Ra$ is not a Bass order, 
then its branch contains 
a ball of positive radius $r$ by 
Theorem \ref{t11} and
Prop. \ref{p22}. 
Assume $r$ is maximal.
Then Prop. \ref{p47} shows that 
$\Ra=\hat{\Ra}^{[r]}$ for some
order $\hat{\Ra}$ that is a ghost intersection 
of maximal orders. The branch of $\Ra$
contains precisely the maximal orders at distance
$r$ from the branch of $\hat{\Ra}$ by 
Prop. \ref{p21}. Since $r$ was chosen maximal, 
the branch of $\hat{\Ra}$ must be the stem
of the branch of $\Ra$, so that is a line
(or vertex), by Prop. \ref{p22}, 
and $\hat{\Ra}$ is a Bass
order. In particular, it is the Gorenstein closure
of $\Ra$.
On the other hand, if the Gorenstein
closure $\tilde{\Ra}$ of $\Ra$ is a Bass order,
then Prop. \ref{p45} shows that $\tilde{\Ra}$
is a ghost intersection of maximal orders, and
so is $\Ra$ by Prop. \ref{p46}. 
The result follows.
\end{proof}

\section{Examples}

\begin{example}\label{e61}
    If $L/K$ is a quadratic extension, and if
    $\Da\subseteq\alge_L$ is a maximal 
    order that is not defined over $K$, then
    it can be proved that $\Da\cap\alge $
    is a Bass order. In fact, if
    $\mathtt{h}\in\Da\cap\alge $,
    then the branch 
    $\mathfrak{s}_K(\mathtt{h})$ is not empty.
    If $L/K$ is unramified, then, since 
    $\mathfrak{s}_L(\mathtt{h})$ is connected,
    it contains the closest vertex $(\Da_1)_L$ in 
    $\mathfrak{t}(L/K)$ to $\Da$, and therefore
    $\Da\cap\alge \subseteq\Da_1$.
    Since $\mathrm{GL}_2(K)$ acts transitively
    on the set $\mathbb{P}^1(L)\smallsetminus
    \mathbb{P}^1(K)$, which corresponds to the
    set of visual limits of $\mathfrak{t}(L)$
    that are not visual limits of
    $\mathfrak{t}(L/K)$, these are precisely 
    the intersections in case \ref{Ea1} in the
    proof of proposition \ref{p41n}, which proves the claim. Assume now that $L/K$ is 
    ramified. Let $\hat{\Da}$ be the vertex of
    $\mathfrak{t}(L/K)$ that is closest to $\Da$.
    This is the barycenter of an edge in $\mathfrak{t}(K)$ with endpoints 
    $\Da_1$ and $\Da_2$. It is proved
    as before that $\mathtt{h}$ is contained
    in either $\Da_1$ or $\Da_2$.
    Assume it is $\Da_1$. We claim that
    $\mathtt{h}$ is also contained
    in $\Da_2$. This gives us 
    $\Da\cap\alge \subseteq\Da_1\cap\Da_2$ and 
    the result follows as before.
    
    Now we prove the claim.
    We can choose a basis on which
    $\Da_1$ is the ring of integral matrices
    and $\Da_2$ is the ring $\Da'$ in 
    (\ref{nba}). Then the order
    $\hat{\Da}$ has the form 
    $\sbmattrix {\oink_L}{\pi_L^{-1}\oink_L}
    {\pi_L\oink_L}{\oink_L}$,
    so that $\mathtt{h}\in\hat{\Da}$
    implies that the lower left coefficient
    of $\mathtt{h}$ is not a unit, and the result follows. 
\end{example}

\begin{example}
    Let $E/K$ be an unramified cubic extension.
    Let $\Da$ be a vertex in $\mathfrak{t}(E)$
    at distance $1$ from the closest vertex
    $\Da_1$ in $\mathfrak{t}(E/K)$. Let
    $\mathtt{h}\in\Da\cap\alge $.
    Again we have $\mathtt{h}\in\Da_1$.
    Since it is also contained in the
    neighbor $\Da$ of $(\Da_1)_E$,
     the image of $\mathtt{h}$ in
    $(\Da_1)_E/\pi_E(\Da_1)_E$ has an eigenvalue
    corresponding to that edge. Since
    a two-by-two matrix cannot gain a
    new eigenspace by passing to a cubic
    extension, unless it is an scalar matrix,
    we conclude that $\Da\cap\alge =
    \Da_1^{[1]}$.
\end{example}

\begin{example}\label{e63}
    Let $F/K$ be a ramified cubic extension.
    Let $\Da$ be a vertex in $\mathfrak{t}(F)$
    at distance $1$ from a vertex $\hat{\Da}$
    in $\mathfrak{t}(F/K)$ that is not a vertex of
    $\mathfrak{t}(K)$, as in Fig.
    \ref{f5nn}.\textbf{I}.  
    Let $\mathtt{h}\in\Ba=
    \Da\cap\alge $.
    It is proved as in Ex. \ref{e61} that
    $\mathtt{h}$ is contained in both
    endpoints $\Da_1$ and $\Da_2$, which are defined over $K$. To simplify notations,
    we write $\Da_1$ and $\Da_2$ for the maximal
    orders in $\alge$, while using
    $(\Da_1)_F$ and $(\Da_2)_F$ for its extensions to $F$.
    Note that $\Ea=\Da_1\cap\Da_2$ is an Eichler
    order, so it contains a nontrivial 
    idempotent. Therefore, the same holds for
    $\Ea_F$. In particular $\Ea_F$ is an Eichler order and therefore, by Cor. \ref{c3x}
    it must be the intersection
    $(\Da_1)_F\cap(\Da_2)_F$. We conclude
    that this order is not contained
    in $\Da$. Furthermore, the branch of
    $\Ba'=\Da\cap(\Da_1)_F\cap(\Da_2)_F\subseteq
    \alge_F$ is a tubular neighborhood
    of width $1$ of a line of length 1
    (c.f. Ex. \ref{e24}). Note that
    $\Ba'\cap\alge =\Ba$.
    In particular, $\mathtt{h}\in\Ba'$.
    \begin{figure}
    \centering
    \unitlength 1mm 
\linethickness{0.4pt}
\ifx\plotpoint\undefined\newsavebox{\plotpoint}\fi 
\begin{picture}(17,18)(0,18)
\put(2,29){\line(1,0){16}}
\put(1,28){$\bullet$}\put(17,28){$\bullet$}
\put(4.5,28.5){${}^{\hat{\Da}}$}
\put(-1,29){${}^{\Da_1}$}
\put(15,29){${}^{\Da_2}$}
\put(7,29){\line(0,-1){1}}
\put(7,27){\line(0,-1){1}}
\put(7,25){\line(0,-1){1}}
\put(7,23){\line(0,-1){1}}
\put(6,28){$*$}\put(6,21){$*$}
\put(3.5,20){${}^{\Da}$}
\put(4.5,28.5){${}^{\hat{\Da}}$}
\put(13,29){\line(0,-1){1}}
\put(13,27){\line(0,-1){1}}
\put(13,25){\line(0,-1){1}}
\put(13,23){\line(0,-1){1}}
\put(12,28){$*$}\put(12,21){$*$}
\put(09,33){$\textnormal{\textbf{I}}$}
\end{picture}
\qquad
    \unitlength 1mm 
\linethickness{0.4pt}
\ifx\plotpoint\undefined\newsavebox{\plotpoint}\fi 
\begin{picture}(17,18)(0,18)
\put(4,29){\line(1,0){12}}
\put(3,28){$\bullet$}\put(15,28){$\bullet$}
\put(4,29){\line(0,1){6}}
\put(4,32){\line(1,0){3}}
\put(16,29){\line(0,1){6}}
\put(16,32){\line(1,0){3}}
\put(10,29){\line(0,-1){1}}
\put(10,27){\line(0,-1){1}}
\put(10,25){\line(0,-1){1}}
\put(10,23){\line(0,-1){1}}
\put(9,28){$*$}\put(9,21){$*$}
\put(6,20){${}^{\Da}$}
\put(09,33){$\textnormal{\textbf{II}}$}
\put(-1,28){$\cdots$}\put(17,28){$\cdots$}
\end{picture}
\qquad
    \unitlength 1mm 
\linethickness{0.4pt}
\ifx\plotpoint\undefined\newsavebox{\plotpoint}\fi 
\begin{picture}(17,18)(0,18)
\put(-2,29){\line(1,0){24}}
\put(9,28){$\bullet$}
\put(1,29){\line(0,1){3}}
\put(4,29){\line(0,1){6}}
\put(4,32){\line(1,0){3}}
\put(16,29){\line(0,1){6}}
\put(16,32){\line(-1,0){3}}
\put(19,29){\line(0,1){3}}
\put(10,29){\line(0,-1){12}}
\put(10,23){\line(1,0){6}}
\put(13,23){\line(0,1){3}}
\put(10,20){\line(1,0){3}}
\put(7,29){\line(0,-1){1}}
\put(7,27){\line(0,-1){1}}
\put(7,25){\line(0,-1){1}}
\put(7,23){\line(0,-1){1}}
\put(6,28){$*$}\put(6,21){$*$}
\put(3.5,20){${}^{\Da}$}
\put(8,33){$\textnormal{\textbf{III}}$}
\end{picture}
    \caption{The configurations in
    Ex. \ref{e63}, Ex. \ref{e66} and
    Ex. \ref{e67}.}
    \label{f5nn}
\end{figure}
    
    Let $L/K$ be a ramified quadratic extension,
    and let $E=FL$, which is a totally ramified
    extension of degree $6$ over $K$. Then the
    branch of $\Ba'_E$ is a tubular neighborhood
    of width $2$ of a line of length 2. This implies that the branch 
    $\mathfrak{s}_E(\mathtt{h})$ contains
    the midpoint $(\widetilde{\Da})_E$ 
    of the line
    from $(\Da_1)_E$ to $(\Da_2)_E$. Note
    that the midpoint is the extension of an order
    $\widetilde{\Da}\subseteq\alge_L$. 
    Since $\mathtt{h}$ is also contained
    in all the neighbors of 
    $(\widetilde{\Da})_E$,
    its image in  
    $(\widetilde{\Da})_E/
    \pi_E(\widetilde{\Da})_E$
    is an scalar matrix. This implies
    that the same hold over $L$, i.e., 
    $\mathtt{h}$ is contained in all the
    neighbors of $\widetilde{\Da}$. In particular, 
    it is contained in any order of the form
    $\Ba''=\mathcal{O}_{\lu}+\boldsymbol\pi_{\lu}
    \Ea$ with $\lu\cong L$. On the other hand,
    if $\mathtt{h}'$ is an element in
    $\Ba''$, its branch 
    $\mathfrak{s}_L(\mathtt{h}')$ contains
    $(\Da_1)_L$, $(\Da_2)_L$ and another order
    $\Da'$ at distance $1$ from the midpoint
    $\widetilde{\Da}$. Then $\mathfrak{s}_F(\mathtt{h}')$ contains
    $(\Da_1)_F$, $(\Da_2)_F$ and $\Da'_F$.
    In particular, again reasoning as in Ex.
    \ref{e24}, it contains a ball that includes
    the vertex $\Da_F$. We conclude that $\Da$ is a vertex in
    $\mathfrak{s}_E(\mathtt{h}')$, whence $\Ba=\Ba''$.
\end{example}

\begin{example}\label{e64}
    Consider the order  $\Ba''=\mathcal{O}_{\lu}+
    \boldsymbol\pi_{\lu}\Ea$ in the preceding example. It seems to 
    depend on the ramified extension
    $L/K$, but the computations in
    that example show that
    this is not the case. Note that,
    when $K$ is not dyadic, this order
    is in the isomorphism class corresponding to the vertex $w$
    in Fig. \ref{nbh1}.
\end{example}

\begin{example}
    Let $\Ba$ be the order generated
    by the image of a faithful representation
    of the dihedral group $D_4$ into
    $\mathbb{M}_2(\mathbb{Q}_2)$. 
    It follows from \cite[Ex. 9.1]{AAS}
    that $\Ba_{\mathbb{Q}_2(\sqrt{-5})}$ is 
    contained in precisely
    the maximal orders in a ball of
    radius $1$ whose
    center is the barycenter of an edge
    in $\mathfrak{t}(\mathbb{Q}_2)$.
    In particular, its branch 
    $\mathfrak{s}_{\mathbb{Q}_2}(\Ba)$
    is an edge.
    we conclude that $\Ba$ is a Bass order
    of the form 
    $\mathcal{O}_{\lu}+\boldsymbol\pi_{\lu}\Ea$, 
    where $\lu$ is isomorphic
    to $\mathbb{Q}_2(\sqrt{-5})$ and
    $\Ea$ is an Eichler order.
\end{example}

\begin{example}\label{e66}
    Let $\mathtt{h}$ be a matrix with
    eigenvalues in $K$, so that its
    branch $\mathfrak{s}_K(\mathtt{h})$
    is a tubular neighborhood of width $s$
    of a maximal path. Fix a ramified quadratic extension $L/K$, and 
    consider the order $\Ba_r=
    \mathcal{O}_{\lu}+
    \boldsymbol\pi_{\lu}^r\Ea$. To find the
    values of $r$ for which $\Ba_r$
    contains a conjugate of $\mathtt{h}$,
    we can just look for the largest
    value of $r$ for which
    $\mathfrak{s}_L(\mathtt{h})$
    contains a vertex at distance $r$
    from the barycenter of the corresponding edge.
    Note that the branch $\mathfrak{s}_L(\mathtt{h})$
    is a tubular neighborhood of width $2s$ of a maximal path.
    Since the edge can be placed
    on the stem, as in Fig. \ref{f5nn}.\textbf{II}, we get
    the inequality $r\leq2s$.
\end{example}

\begin{example}\label{e67}
    Assume now that $\mathtt{h}$ has eigenvalues in an 
    unramified quadratic extension of $K$, so that its
    branch $\mathfrak{s}_K(\mathtt{h})$
    is a tubular neighborhood of width $s$
    of a vertex. The reasoning here is
    analogous to the preceding example,
    except that we no longer have stem
    edges to use, so the optimal placement is the 
    one depicted in Fig. \ref{f5nn}.\textbf{III}.
    We get the inequality $r+1\leq2s$. 
    This also work for matrix families,
    when the intersection of the corresponding
    branches is a ball. This case appear often
    in practice, since the intersection
    of two tubular neighborhoods of lines
    is a ball, whenever the distance $e$ between
    the stems is larger than the difference $d$
    of the widths and satisfies $d\equiv e\mod 2$.
    See \cite[Prop. 2.3]{AS16}.
\end{example}

\begin{example}\label{e69}
    Fix a maximal order $\Da_1$ in $\alge$.
    To compute the number of orders of the form
    $\mathcal{O}_{\lu}+
    \pi_K^r\Da_1$ with branch $\{\Da_1\}$, where
    $\lu$ runs over the set of two
    dimensional subalgebras of $\alge$ that are 
    isomorphic to a fixed unramified quadratic extension $L$ 
    of $K$, we can simply count the appropriate
    paths in the Bruhat-Tits tree 
    $\mathfrak{t}(L)$.
    In figure \ref{f6}.\textbf{I} 
    the orders $\Da'$ and $\Da''$
    correspond to maximal orders in
    $\alge_{L}$ that fail two contain a generator
    of any such  order $\mathcal{O}_{\lu}$
    that is contained in $\Da\cap\alge$,
    since $\mathfrak{s}_L(\mathcal{O}_{\lu})$
    is a line.
    We conclude that these maximal orders
    define Bass orders that are isomorphic
    to $\Ba=\mathcal{O}_{\lu}+
    \pi_K^r\Da_1$, but do not coincide with it.
    In fact, the only other maximal order
    at distance $r$
    defining the same Bass order as $\Da$
    is the other vertex at the same distance
    in the line 
    $\mathfrak{s}_L(\mathcal{O}_{\lu})$.
    This is the Galois image $\sigma(\Da)$,
    where $\sigma$ is a generator of the Galois 
    group of the extension $L/K$. A similar 
    argument work for ramified extensions,
    but the branches that we need to count
    are no longer lines. For instance, 
    an order of the form 
    $\mathcal{O}_{\lu}$, where
    $\lu$ is isomorphic to an unramified 
    extension of a non-dyadic field, is contained
    in the endpoints $\Da_1$ and $\Da_2$ of the 
    edge in $\mathfrak{t}(K)$ whose barycenter
    $\hat{\Da}$ is in the stem of the branch
    $\mathfrak{s}_L(\mathcal{O}_{\lu})$.
    This implies that the latter branch is a tubular neighborhood of width $1$ of its
    stem. In particular, the vertices
    $\Da$, $\Da'$ and $\Da''$ in Fig.
    \ref{f6}.\textbf{II} define the same
    Bass order. We need, therefore to count
    the stems of the branches, which, for
    the order $\Da$, would be the line between
    $\widetilde{\Da}$ and its Galois image
    $\sigma(\widetilde{\Da})$.
\begin{figure}
    \centering
    \unitlength 1mm 
\linethickness{0.4pt}
\ifx\plotpoint\undefined\newsavebox{\plotpoint}\fi 
\begin{picture}(30,24)(30,14)
\put(46,14){\textbf{I}}
\put(48,20){\line(0,1){8}}
\put(47,23){$\bullet$}\put(49,24){${}^{\Da_1}$}
\put(60,24){\line(-1,0){1}}
\put(58,24){\line(-1,0){1}}
\put(56,24){\line(-1,0){1}}
\put(54,24){\line(-1,0){1}}
\put(52,24){\line(-1,0){1}}
\put(50,24){\line(-1,0){1}}
\put(48,24){\line(-1,0){1}}
\put(46,24){\line(-1,0){1}}
\put(44,24){\line(-1,0){1}}
\put(42,24){\line(-1,0){1}}
\put(40,24){\line(-1,0){1}}
\put(38,24){\line(-1,0){1}}
\put(36,24){\line(-1,0){1}}
\put(34,23.25){$*$}\put(32,23){${}^{\Da}$}
\put(60,23.25){$*$}\put(62,23){${}^{\sigma(\Da)}$}
\put(42,20){\line(0,1){1}}
\put(42,22){\line(0,1){1}}
\put(42,24){\line(0,1){1}}
\put(42,26){\line(0,1){1}}
\put(54,20){\line(0,1){1}}
\put(54,22){\line(0,1){1}}
\put(54,24){\line(0,1){1}}
\put(54,26){\line(0,1){1}}
\put(57,20){\line(0,1){1}}
\put(57,22){\line(0,1){1}}
\put(57,24){\line(0,1){1}}
\put(57,26){\line(0,1){1}}
\put(38,28){\line(0,-1){1}}
\put(38,26){\line(0,-1){1}}
\put(38,24){\line(0,-1){1}}
\put(38,22){\line(0,-1){1}}
\put(37,19){$*$}\put(34,16){${}^{\Da'}$}
\put(37,27){$*$}\put(34,27){${}^{\Da''}$}
\end{picture}
\qquad\qquad
\linethickness{0.4pt}
\ifx\plotpoint\undefined\newsavebox{\plotpoint}\fi 
\begin{picture}(30,24)(30,14)
\put(46,14){\textbf{II}}
\put(48,20){\line(0,1){8}}
\put(47,19){$\bullet$}\put(49,19){${}^{\Da_1}$}
\put(47,27){$\bullet$}\put(49,27){${}^{\Da_2}$}
\put(47,23){$*$}\put(49,23){${}^{\hat{\Da}}$}
\put(60,24){\line(-1,0){1}}
\put(58,24){\line(-1,0){1}}
\put(56,24){\line(-1,0){1}}
\put(54,24){\line(-1,0){1}}
\put(52,24){\line(-1,0){1}}
\put(50,24){\line(-1,0){1}}
\put(48,24){\line(-1,0){1}}
\put(46,24){\line(-1,0){1}}
\put(44,24){\line(-1,0){1}}
\put(42,24){\line(-1,0){1}}
\put(40,24){\line(-1,0){1}}
\put(38,24){\line(-1,0){1}}
\put(36,24){\line(-1,0){1}}
\put(34,23.25){$*$}\put(32,23){${}^{\Da}$}
\put(37,23.25){$*$}\put(38.5,23.5){${}^{\widetilde{\Da}}$}
\put(56,23.25){$*$}\put(57.25,23.75){${}^{\sigma(\widetilde{\Da})}$}
\put(42,20){\line(0,1){1}}
\put(42,22){\line(0,1){1}}
\put(42,24){\line(0,1){1}}
\put(42,26){\line(0,1){1}}
\put(54,20){\line(0,1){1}}
\put(54,22){\line(0,1){1}}
\put(54,24){\line(0,1){1}}
\put(54,26){\line(0,1){1}}
\put(57,20){\line(0,1){1}}
\put(57,22){\line(0,1){1}}
\put(57,24){\line(0,1){1}}
\put(57,26){\line(0,1){1}}
\put(38,28){\line(0,-1){1}}
\put(38,26){\line(0,-1){1}}
\put(38,24){\line(0,-1){1}}
\put(38,22){\line(0,-1){1}}
\put(37,19){$*$}\put(34,16){${}^{\Da'}$}
\put(37,27){$*$}\put(34,27){${}^{\Da''}$}
\end{picture}
\qquad\qquad
\unitlength 1mm 
\linethickness{0.4pt}
\ifx\plotpoint\undefined\newsavebox{\plotpoint}\fi 
\begin{picture}(30,18)(30,16)
\put(60,20){\line(0,1){8}}
\put(46,16){\textbf{III}}
\put(59,19){$\bullet$}\put(61,19){${}^{\Da_1}$}
\put(59,27){$\bullet$}\put(61,27){${}^{\Da_2}$}
\put(59,23){$*$}\put(61,23){${}^{\hat{\Da}}$}
\put(60,24){\line(-1,0){1}}
\put(58,24){\line(-1,0){1}}
\put(56,24){\line(-1,0){1}}
\put(54,24){\line(-1,0){1}}
\put(52,24){\line(-1,0){1}}
\put(50,24){\line(-1,0){1}}
\put(48,24){\line(-1,0){1}}
\put(46,24){\line(-1,0){1}}
\put(44,24){\line(-1,0){1}}
\put(42,24){\line(-1,0){1}}
\put(40,24){\line(-1,0){1}}
\put(38,24){\line(-1,0){1}}
\put(36,24){\line(-1,0){1}}
\put(41,23.25){$*$}
\put(41,19)
{${}^{\sigma\left(\widetilde{\Da}\right)}$}
\put(33,23.25){$*$}\put(29,24){${}^{\sigma(\Da)}$}
\put(48,24){\line(0,1){1}}
\put(48,26){\line(0,1){1}}
\put(48,28){\line(0,1){1}}
\put(48,30){\line(0,1){1}}
\put(48,32){\line(0,1){1}}
\put(48,34){\line(0,1){1}}
\put(48,36){\line(0,1){1}}
\put(48,38){\line(0,1){1}}
\put(47,29.25){$*$}\put(45,29){${}^{\widetilde{\Da}}$}
\put(47,39.25){$*$}\put(49,39){${}^{\Da}$}
\put(42,24){\line(0,1){1}}
\put(42,26){\line(0,1){1}}
\put(42,28){\line(0,1){1}}
\put(42,30){\line(0,1){1}}
\put(42,28){\line(-1,0){1}}
\put(40,28){\line(-1,0){1}}
\put(38,24){\line(0,-1){1}}
\put(38,22){\line(0,-1){1}}
\put(48,30){\line(1,0){1}}
\put(50,30){\line(1,0){1}}
\put(52,30){\line(1,0){1}}
\put(54,30){\line(1,0){1}}
\put(48,36){\line(-1,0){1}}
\put(46,36){\line(-1,0){1}}
\put(52,30){\line(0,1){1}}
\put(52,32){\line(0,1){1}}
\end{picture}
\caption{The configuration of orders in
Ex. \ref{e69}.}
    \label{f6}
\end{figure}    
The same holds in the dyadic case except
that, here, the stem of 
$\mathcal{O}_{\lu}$ no longer
contains the barycenter $\hat{\Da}$.
For instance, in Fig. \ref{f6}.\textbf{III},
the stem is the line of length $2$ between
$\widetilde{\Da}$ and its Galois image. In the 
picture, assuming a residual field with two 
elements, this line is unique. However, this is no longer the case if $\Da$ is at distance $5$
or more from $\hat{\Da}$. In principle this
method can be applied to study orders of the 
form $\Ba=\oink_{\lu}+\boldsymbol\pi_{\lu}^r\Ea$,
for different ramified quadratic
extensions $L/K$, by passing
to a sufficiently large field.
\end{example}

\section{Acknowledgments}

The author would like to thanks Dr. Claudio Bravo for several useful suggestions.

  Luis Arenas-Carmona\\
Universidad de Chile,\\ Facultad de Ciencias, \\Casilla 653, Santiago,\\
   Chile \\
   \email{learenas@u.uchile.cl}

\begin{thebibliography}{99}

\bibitem{AAS}
{\scshape B. Aguil\'o-Vidal}, 
{\scshape L. Arenas-Carmona} and
{\scshape M. Saavedra-Lagos}, `Two dimensional 
integral representations via branches of the
Bruhat-Tits tree'.
{\em   Arxiv Math.} \textbf{2506.05661v1} (6 Jun 2025).





\bibitem{Eichler2}
{\sc L. Arenas-Carmona}, `Trees, branches, and
spinor genera',  {\em Int. J. Number T.} \textbf{9} (2013), 1725-1741. 



\bibitem{AS16}
{\sc L. Arenas-Carmona} and {\sc I. Saavedra}, 
'On some branches of the Bruhat-Tits tree', 
 {\em Int. J. Number T.} \textbf{12} (2016), 
 813-831.


\bibitem{GB1}
{\sc L. Arenas-Carmona} and {\sc C. Bravo},
`Computing embedding numbers and branches of 
orders via extensions of the Bruhat-Tits tree'. 
 {\em Int. J. Number T.}  
{\bf 15} (2019), 2067-2088. 



\bibitem{Brz82}
{\sc J. Brzezinski},  (1983). `A characterization of Gorenstein orders in quaternion algebras'. {\em Math. Scand.} \textbf{50} (1982), 19-24.

\bibitem{Brz83}
{\sc J. Brzezinski},  (1983). `On orders in quaternion algebras'. {\em Comm. Algebra} \textbf{11} (1983), 501-522.


\bibitem{Chan}
{\sc W.K. Chan} and {\sc F. Xu}, 'On representations of
spinor genera',  {\em  Compositio Math.} \textbf{140}
(2004), 287-300.


\bibitem{Lem}
{\sc  S. Lemurel},
 'Quaternion orders and ternary quadratic forms'.
 {\em   Arxiv Math.} \textbf{1103.4922v1} (25 Mar 2011).













\bibitem{trees}
{\sc J.-P. Serre},  {\em Trees}, Springer Verlag, Berlin, 1980.

\bibitem{Tu}
{\sc F.-T. Tu}, `On orders of M(2,K) over a non-archimedean local field',  {\em Int. J. Number T.} \textbf{7} (2011), 1137-1149.

 \bibitem{voight}
 {\sc J. Voight}, {\em Quaternion Algebras}, Grad. Texts in Math.
 \textbf{288},
Springer Verlag, Cham, 2021.



\end{thebibliography}
\end{document}